\documentclass[reqno,10pt,draft,a4paper]{amsart}
\usepackage{amsthm}
\usepackage{amsmath}
\usepackage{amssymb}
\usepackage{latexsym}
\usepackage{stmaryrd}
\usepackage{epsf}
\usepackage{booktabs}
\usepackage{graphics}
\renewcommand{\includegraphics}{\epsfbox}

\usepackage[usenames]{color}

\usepackage[all]{xy}

\UseComputerModernTips
\CompileMatrices

\usepackage{multicol}

\renewcommand{\ge}{\geqslant}
\renewcommand{\le}{\leqslant}

\newcommand{\C}{\mathbb{C}}

\newcommand{\N}{\mathbb{N}}

\newcommand{\dd}{\delta}
\newcommand{\dl}{\delta_L}
\newcommand{\dr}{\delta_R}
\newcommand{\kl}{\kappa_L}
\newcommand{\kr}{\kappa_R}
\newcommand{\kk}{\kappa_{LR}}

\newcommand\vstrut[2]{\rule[-#1]{0pt}{#2}}

\newcommand{\ldescent}[1]{{\mathcal L (#1)}}
\newcommand{\rdescent}[1]{{\mathcal R (#1)}}
\newcommand\idest{i.e.,\ }
\newcommand\br{{\mathbf r}}
\newcommand\bs{{\mathbf s}}
\newcommand\bu{{\mathbf u}}
\newcommand\bv{{\mathbf v}}

\begin{document}
\theoremstyle{plain}
\numberwithin{subsection}{section}
\newtheorem{thm}{Theorem}[section]
\newtheorem{prop}[thm]{Proposition}
\newtheorem{cor}[thm]{Corollary}
\newtheorem{clm}[thm]{Claim}
\newtheorem{lem}[thm]{Lemma}
\newtheorem{conj}[thm]{Conjecture}
\theoremstyle{definition}
\newtheorem{defn}[thm]{Definition}
\newtheorem{rem}[thm]{Remark}
\newtheorem{eg}[thm]{Example}

\title{
A presentation for
the symplectic blob  algebra }
\author{R. M. Green \and  P. P. Martin
\and A. E. Parker$^1$} 
\address{Department of Mathematics \\ University of Colorado \\
Campus Box 395 \\ Boulder, CO  80309-0395 \\ USA }
\email{rmg@euclid.colorado.edu} 
\address{Department of Mathematics \\ University of Leeds \\ Leeds,
  LS2 9JT \\ UK}
\email{ppmartin@maths.leeds.ac.uk}
\email{parker@maths.leeds.ac.uk}
\footnotetext[1]{Corresponding author}

\begin{abstract}
The symplectic blob algebra $b_n$ ($n \in \N$) is a
finite dimensional algebra defined by a
multiplication rule on a basis of certain diagrams.
The rank
$r(n)$ of $b_n$ is not known in general, but 
 $r(n)/n$ grows unboundedly with $n$.
For each $b_n$ we define an algebra by presentation, such that the
number of generators and relations grows linearly with $n$. 
We prove that
these algebras are isomorphic.
\end{abstract}

\maketitle

\section{Introduction}\label{intro}

\newcommand{\bc}{boundary condition}
\newcommand{\TL}{Temperley--Lieb}
\newcommand{\SM}{Statistical Mechanics}
The transfer matrix formulation of lattice Statistical Mechanics 
(see e.g. \cite{Baxter,marbk}) is a source for many sequences of
algebras and representations --- among the best known examples are the
Temperley--Lieb algebras \cite{TL} and the quantum groups
\cite{Jimbo85}. 
Physically one seeks to diagonalise the transfer matrix, and this
corresponds to computing the irreducible representations of the
associated algebras. 
Statistical Mechanics often provides algebras with a basis of
`diagrams' (describing the configuration of physical states),
leading to the notion of diagram algebras.
The Temperley--Lieb diagram algebra arises in several different
Statistical Mechanical models (such as Potts models, $q$-spin chains
and vertex models), but in each case the algebra manifests only when
specific `open' physical \bc s are imposed.
It is physically appropriate to consider other \bc s, 
however, and this forces
a generalisation in the algebra. For example, periodic boundary
conditions necessitate generalisation to the blob diagram algebra
\cite{martsaleur}. 
More recently it has been shown \cite{degiernichols} that other physically
interesting \bc s necessitate further generalisation.
Both the \TL\  and blob algebras have alternative definitions by
presentation, and each of the diagram- and presentation-based
definitions suggest candidates for suitable generalisations.
The study of these two generalisations has begun in  \cite{degiernichols,DegierPyatov04}
and  \cite{gensymp}, but the isomorphism between them was not
established
(and it does not follow from the isomorphisms for the earlier algebras). 
We prove the isomorphism here.

In the study of Hecke algebras of arbitrary type, a useful tool is the
\TL\ algebra of the same type (see \cite{GrahamLehrer03,gensymp} for references). 
This is, in each case, a
Hecke quotient 
algebra defined by presentation. Type-$A$ gives the presentational
form of the ordinary \TL\ algebra. Type-$B$ gives the blob algebra;
and the presentational form of the new generalisation is a quotient of
type-$\tilde C$
(also known as the two-boundary \TL\ algebra \cite{degiernichols}). 
For this reason, the new diagram algebra is known as
the  \emph{symplectic blob algebra}, $b_n$
(in \cite{gensymp} the notation $b_n^x$  is used).

\medskip

In \cite{gensymp} we 
investigated its generic representation theory and proved various
representation theoretically
important properties of the algebra, for instance that it has a
cellular basis, that it is generically semi-simple
(in the Hecke algebra parameters),
that the associated sequence $n \rightarrow \infty$ of module
categories has a `thermodynamic limit',
and that it is a quotient of the Hecke algebra of type-$\tilde{C}$. 
For a number of reasons explained in the original paper (the role of
Temperley--Lieb and blob algebras in Statistical Mechanics
and in solving the Yang--Baxter equations; the
intrinsic interest in the Hecke algebra of type-$\tilde{C}$,
and so on) one is interested in the representation theory of this
algebra.
The representation theory of the ordinary \TL\ and blob cases 
is rather well understood, 
and has an elegant geometrical description,
over an arbitrary algebraically closed field \cite{blobcgm}.
So far here, however,  not even the blocks over $\C$ are known.
As with finite-dimensional algebras defined as diagram algebras in general
(or indeed any algebra), a powerful
tool in representation theory is to be able to give 
an efficient presentation,
so this is our objective here.

The paper is structured as follows. We first review the various
objects and notations and some of the basic properties of the
symplectic blob algebra that will be used in the paper.
This is followed by a statement and proof of a presentation for the
algebra. The proof occupies the majority of the paper. 

It is easy to establish an explicit surjective algebra homomorphism in
one direction, and we start with this. However a suitable closed
formula for the rank at level $n$ 
is not presently known for either algebra, 
so we are motivated to use a method that does not rely on rank bounds.
Our method generalises an approach in \cite{G35}, and so should be of
wider interest in the study of Coxeter groups and related algebras.

\section{The symplectic blob algebra} \label{sect:symp}

We start with a summary of \cite[\S6]{gensymp}.
Fix $n,m \in \N$, with $n+m$ even, and $k$ a field. 
A \emph{Brauer $(n,m)$-partition} $p$ is a partition of the set
$V \cup V'$ into pairs, where $V = \{ 1, 2, \ldots, n\}$
and  $V' = \{ 1', 2', \ldots, m'\}$.
Following Brauer \cite{brauer} and Weyl \cite{weyl46}
we will depict $p$ as a \emph{Brauer $(n,m)$-diagram}.
A diagram for $p$ is 
%% by taking 
a rectangle with
$n$ vertices labelled $1$ through to $n$ on the top edge and
$m$ vertices labelled $1'$ through to $m'$ on the bottom, and
 two vertices $a$  and $b$ connected, with an arbitrary line embedded in the
plane of the rectangle, if $\{a,b\} \in p$.

Any two rectangles with embeddings coding 
 the same set partition are called {\em equivalent}, and regarded as
the same Brauer diagram.

Now consider a diagram among whose 
embeddings (in the above sense) are
embeddings with no lines crossing. 
For such a diagram, we may consider 
the sub-equivalence class of embeddings  that
indeed have no crossings. This class (or a representative thereof) is a
\emph{Temperley--Lieb diagram}.
Note that such a diagram $d$ defines not only a pair-partition of $V
\cup V'$ but also a partition of the open intervals of the frame of
the rectangle excluding $V \cup V'$ (two intervals are in the same
part if there is a path from one to the other in the rectangle
that does not cross a line of $d$).
 
Our first objective is to define a certain diagram category, that is a
$k$-linear category whose hom-sets each have a basis
consisting of diagrams, and where
multiplication is defined by 
diagram concatenation
(the object class is $\N$ in our case), 
and 
simple 
\emph{ straightening rules} to be applied when the concatenated object is not
formally a diagram. 
For example in 
%our case, 
the Brauer or \TL\ diagram category,
a concatenation may produce a diagram, as here:
$$ \epsfbox{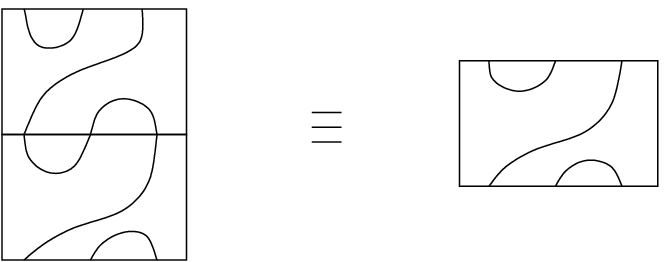}$$
or not, as here:
\begin{equation} \label{eq:1} %$$
 \epsfbox{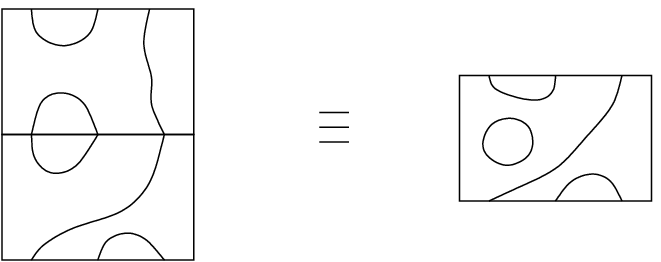}
\end{equation} %$$
A straightening rule is a way of expressing such products 
as (\ref{eq:1})   in the span
of basis diagrams. 

The resulting diagram in (\ref{eq:1}) 
is an example of a \emph{pseudo Temperley--Lieb
diagram}.
Simply put, it fails to be a proper Temperley--Lieb
diagram because of the loop.
The  set of pseudo Temperley--Lieb diagrams includes all the
Temperley--Lieb diagrams, but we also allow diagrams with loops, which
may appear anywhere in the diagram, although still with no crossing
lines. 
Here
(in addition to the equivalence of different embeddings of open lines, as before)
 isotopic deformation of a loop
without crossing a line
results in an equivalent embedding.

The set of pseudo Temperley--Lieb diagrams with $m=n$
is closed under concatenation. 
Thus we can define a  straightening rule for multiplication of Temperley--Lieb
diagrams by
imposing a relation on the $k$-space spanned by pseudo
Temperley--Lieb diagrams that will remove the loops (and is consistent
with concatenation).

\begin{defn}
For $\delta\in k$ and $n \in \N$,
the \emph{Temperley--Lieb algebra} $TL_n = TL_n(\delta)$ 
is the $k$-algebra
with $k$-basis the Temperley--Lieb $(n,n)$-diagrams and multiplication defined
by concatenation. We impose the relation: each loop that may arise
when multiplying is omitted and 
replaced by a factor $\delta$.
\end{defn}

Next we generalise to \emph{decorated Temperley--Lieb} diagrams. 
Here we put
elements of a monoid on the lines (like beads on a string). 
When decorated diagrams are concatenated,
two or more line segments are combined in sequence as before. 
But now we need a rule to combine the monoid
elements on these segments to make a new monoid element for the
combined line.
One such rule is simply to multiply in the monoid in the indicated order.
This gives us a well  defined associative 
diagram calculus --- see section 3 of \cite{gensymp}
for a detailed discussion and proof of this. 

We will now focus on a particular set of decorated Temperley--Lieb diagrams
--- the ones used to define the symplectic blob algebra.
To begin, we decorate with 
the
free monoid on two generators. 
The beads depicting these generators are called {\em blobs}: 
a ``left'' blob, $L$, (usually a black filled-in circle 
on the diagrams) and a
``right'' blob, $R$, (usually a white filled-in circle on the diagrams).
 
A line in a (pseudo) Temperley--Lieb diagram is said to be
\emph{$L$-exposed} (respectively \emph{$R$-exposed}) if it can be deformed to touch
the left hand side (respectively right hand side) of the 
rectangular frame without
crossing any other lines.

A \emph{left-right blob pseudo-diagram} is a diagram obtained
from a  pseudo Temperley--Lieb
diagram by allowing left and
right blob decorations with the following constraints.
Any line decorated with a left blob must be $L$-exposed and
any line decorated with a right blob must be $R$-exposed.
Also all segments with decorations must
be deformable so that the left blobs can touch the left hand
side and the right blobs touch the right hand side of the frame
{\em simultaneously} without crossing.

Concatenating diagrams cannot change a $L$-exposed line to a
non-$L$-exposed line, and similarly for $R$-exposed lines. Thus the set of
left-right blob pseudo-diagrams is closed under diagram concatenation.
(See \cite[proposition 6.1.2]{gensymp}.)

The set of left-right blob pseudo-diagrams is infinite.
For example, if a left blob can appear on a line in a given
underlying pseudo-diagram,
then arbitrarily many such blobs can appear. 
To define a finite dimensional $k$-algebra, as for
the blob algebra (see \cite[section 1.1]{marwood} for a definition)
and the Temperley--Lieb algebra (defined
above), we will 
straighten by certain rules, into the $k$-span of a finite subset.
For example we may identify a pseudo-diagram with
certain localised features, such as multiple blobs on a line, 
with a scalar multiple of an otherwise identical diagram
with other features in that locale (fewer blobs, possibly
none, on that line). 

We now proceed to define a specific such straightening (i.e. a finite
target set, and a suitable collection of rules). We have six
parameters, $\dd, \dl, \dr, \kl, \kr, \kk = k_L = k_R$, 
which are all 
elements in the base field $k$. 

Consider the set of 
eight features drawn on the left-hand sides of the 
sub-tables of table~\ref{blobtab}.
We define $B_n'$ to be the set of left-right blob
pseudo-diagrams with $n$ vertices 
at the top and $n$ at the bottom of the diagram
that do \emph{not} have features from this set.
\newpage
\begin{multicols}{4}
$$
\begin{tabular}{|c|}
\hline
\vstrut{1.5em}{3.8em}
$\raisebox{-0.2cm}{\epsfbox{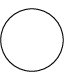}} %undecorated loop 
\mapsto \delta$\\
\hline
\vstrut{2.4em}{5em}
$\raisebox{-0.6cm}{\epsfbox{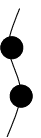}} %two consecutive $L$'s (black blobs)
\mapsto \dl \raisebox{-0.6cm}{\epsfbox{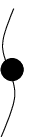}}$ \\%$L$\\
\hline
\end{tabular}
$$\\
$$
\begin{tabular}{|c|}
\hline
\vstrut{2.4em}{5em}
$\raisebox{-0.6cm}{\epsfbox{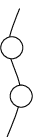}} %two consecutive $R$'s (white blobs)
\mapsto \dr \raisebox{-0.6cm}{\epsfbox{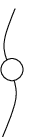}}$ \\%$R$\\
\hline
\vstrut{1.5em}{3.8em}
$\raisebox{-0.2cm}{\epsfbox{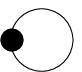}} %loop decorated with an $L$ (black blob)
\mapsto \kl$\\
\hline
\end{tabular}
$$ \\
$$
\begin{tabular}{|c|}
\hline
\vstrut{1.5em}{3.8em}
$\raisebox{-0.2cm}{\epsfbox{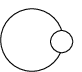}} %loop decorated with an $R$ (white blob)
\mapsto \kr$\\
\hline
\vstrut{1.5em}{3.8em}
$\raisebox{-0.2cm}{\epsfbox{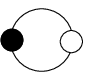}} %loop decorated with a $LR$ pair 
%(black and a white blob)
\mapsto \kk$\\
\hline
\end{tabular}
$$\\
$$
\begin{tabular}{|c|}
\hline
\vstrut{2.4em}{5em}
$\raisebox{-0.6cm}{\epsfbox{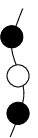}} %an $LRL$ string (black, white, black blob)
\mapsto k_L\raisebox{-0.6cm}{\epsfbox{lline.eps}}$ \\%$L$\\
\hline
\vstrut{2.4em}{5em}
$\raisebox{-0.6cm}{\epsfbox{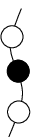}} %an $RLR$ string ( white, black, white blob)
\mapsto k_R\raisebox{-0.6cm}{\epsfbox{rline.eps}}$ \\%$R$\\
\hline
\end{tabular}
$$
\end{multicols}

\begin{table}[ht]
\caption{Table encoding most of the straightening relations for $b_n$.\label{blobtab}}
\end{table}

The set $B_n'$ is finite. We call its elements
\emph{left-right blob diagrams}.

Now define a relation on the $k$-span of all  left-right blob
pseudo-diagrams as follows. If $d,d'$ are scalar multiples of single
diagrams, set
 $d \sim  d'$ if  $d'$ differs from $d$ 
by a substitution from left to right in either sub-table. 
Extend this $k$-linearly. 

A moment's thought makes it clear that to obtain a consistent set of
relations we need $RLRL= k_R RL = k_L RL$, i.e., that $k_L = k_R$.

Another (perhaps longer) moment's thought reveals that 
the $k_L$ relation is only
needed for $n$ odd and the $\kk$ relation is only needed when $n$ is
even.
It turns out to be convenient to set $\kk = k_L=k_R$.

We have the following result.
\begin{prop}[{\cite[section 6.3]{gensymp}}]
The above relations on the $k$-span of left-right blob pseudo-diagrams 
define, with diagram concatenation, 
a finite dimensional algebra,
$b'_n$,
which has a diagram basis $B_n'$. \qed
\end{prop}

We study this algebra by considering the quotient by the 
``topological relation'':
\begin{equation} \label{topquot}
\kappa_{LR} \;\; \raisebox{-0.8cm}{\epsfbox{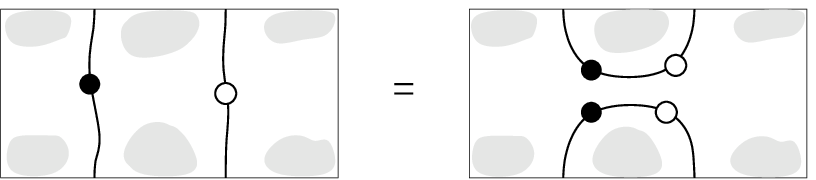}}
\end{equation}
where each
shaded area is shorthand for  subdiagrams that do
not have propagating lines
(a line is called \emph{propagating} if it joins a vertex on the top
of the diagram to one on the bottom of the diagram).
(Note that there is no freedom in choosing the scalar multiple, 
once we require a  relation of this {\em form}.)

We define $B_n$ to be the subset of $B_n'$ that does not
contain diagrams with features as in the right hand side of relation
\eqref{topquot}.

\begin{defn}
We 
define the \emph{symplectic blob algebra}, $b_n$ (or $b_{n}(\dd,
\dl,\dr,\kl,\kr,\kk)$ if we wish to emphasise the parameters)
to be the $k$-algebra with basis $B_{n}$, multiplication defined
via diagram concatenation and relations as in the table above (with
$\kk=k_L=k_R$) and 
with relation \eqref{topquot}.
\end{defn}

That these relations are consistent and that we do obtain an
algebra with basis $B_n$ is proved in \cite[section 6.5]{gensymp}.

We have the following (implicitly assumed in \cite{gensymp}):
\begin{prop}\label{gensblob}
The symplectic blob algebra, $b_{n}$, is
generated by the following diagrams
\begin{multline*} 
e:=\raisebox{-0.4cm}{\epsfbox{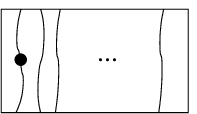}},\  
e_1:=\raisebox{-0.4cm}{\epsfbox{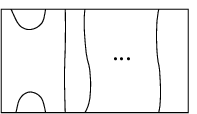}},\ 
e_2:= \raisebox{-0.4cm}{\epsfbox{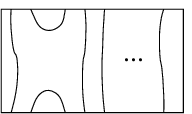}},\ 
\cdots,\ \\
e_{n-1}:=\raisebox{-0.4cm}{\epsfbox{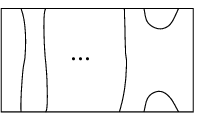}},\ 
f:=\raisebox{-0.4cm}{\epsfbox{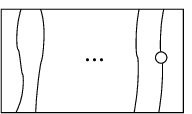}}. 
\end{multline*} 
\end{prop}
\begin{proof}
We may argue in a similar fashion as in appendix A of \cite{gensymp}
but by now inducing on the number of decorations. 
If a diagram $d$ has no decorations then the diagram is a Temperley--Lieb
diagram and the result follows.

So now assume that we have a diagram $d$ with 
$m$ decorations 
and that (for the sake of illustration) that there is a left blob --- we
would use the dual reduction in the case of a right blob.
We claim that we may use the same procedure as in the $l=0$ case
of \cite[appendix A]{gensymp}. 
If there is a decorated line starting in the first position, then we
can decompose the diagram into a product of $e$ then a diagram with one
fewer decoration. If there is no such line then take the first line
decorated with a black blob and do the same reduction as in 
\cite[appendix A]{gensymp}. For example, taking a diagram with a line
with both a left and right blob on it:
$$\epsfbox{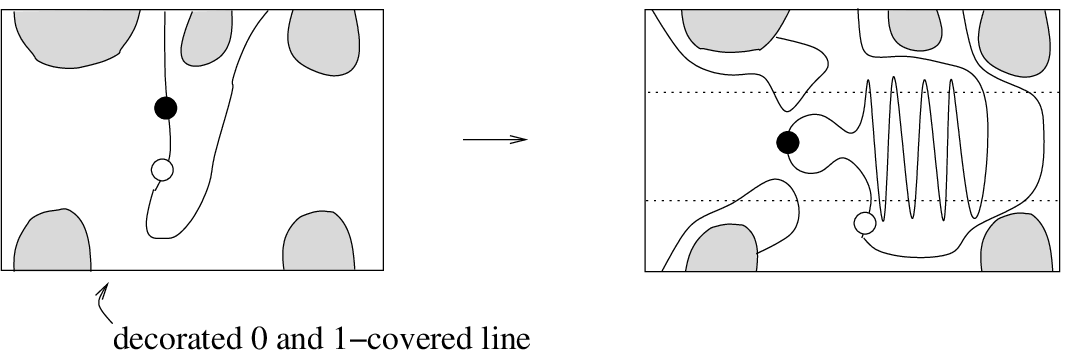}$$
The white blobs can either be moved into the shaded regions or above
or below the horizontal dotted lines. The middle region (after
``wiggling'' the line enough times)  is then the
product $e_1ee_2e_1$. The outside diagrams have strictly fewer
than $m$ decorations and hence the result follows by induction. 
\end{proof}

\section{Presenting the symplectic blob algebra}\label{sect:pres}
\newcommand{\TBTL}{two boundary Temperley--Lieb algebra}
\newcommand{\sblob}{symplectic blob algebra}

We start by defining an algebra by a presentation that is a direct
generalisation of the well-known
presentation for the (ordinary) Temperley--Lieb
algebra. 

\begin{defn}\label{seca5}
Fix $n \ge 1$. 
Let $S_n = \{E_0, E_1, \ldots, E_n\}$, and let $S_n^*$ be the free monoid on $S_n$.
Define the \emph{commutation monoid} $M_n$ to be the quotient of $S_n^*$ by the
relations $$
E_i E_j \equiv E_j E_i \text{ for all } 0 \leq i, j \leq n \text{ with } 
|i - j| > 1
.$$
\end{defn}

\begin{defn}\label{2BTL pres}
Let $P_n = P_n(\dd,\dl,\dr,\kl,\kr,\kk)$ 
be the quotient of the $k$-monoid-algebra of $M_n$  
by the following relations:
\begin{align*}
E_0^2 &= \delta_L E_0,
&
E_1E_0E_1 &= \kappa_L E_1,
\\
E_i^2 &= \delta E_i\quad\mbox{for }1 \le i \le n-1,
&
E_iE_{i+1}E_i &= E_i \quad\mbox{for }1 \le i \le n-2,
\\
E_{n}^2 &= \delta_R E_n,
&
E_{i+1}E_{i}E_{i+1} &= E_{i+1}\quad\mbox{for }1 \le i \le n-2,
\\
&
&
E_{n-1}E_nE_{n-1} &= \kappa_R E_{n-1},
\\
IJI &= \kappa_{LR} I,
&
JIJ &= \kappa_{LR} J,
\end{align*}
where 
$$I= \begin{cases} 
E_1 E_3 \cdots E_{2m-1} &\mbox{if $n=2m$},\\
E_1 E_3 \cdots E_{2m-1} E_{2m+1} &\mbox{if $n=2m +1$},\\
\end{cases}$$
$$J= \begin{cases} 
E_0 E_2 \cdots E_{2m-2} E_{2m} &\mbox{if $n=2m$},\\
E_0 E_2 \cdots E_{2m}  &\mbox{if $n=2m +1$}.\\
\end{cases}$$
\end{defn}
Note $I=E_1$ and $J=E_0$ if $n=1$.
We will sometimes write  $E$ for $E_0$ and $F$ for $E_n$.

\begin{rem}\label{unnumb}
The presentation obtained by omitting the last two relations 
($IJI=\kk I$ and $JIJ=\kk J$) generalises the presentation for the
blob algebra (sometimes known as the one-boundary  Temperley--Lieb
algebra, because of its role in modelling 
two-dimensional Statistical Mechanical
systems with variable boundary conditions at one boundary). For this reason the 
generalisation is sometimes known as
 the two-boundary Temperley--Lieb
algebra. It also coincides,
in Hecke algebra representation theory, with a 
Temperley--Lieb algebra of type-$\tilde{C}$ (see \cite{gensymp}). 
\end{rem}

\begin{thm}\label{thmmain}
Suppose $\dd,\dl,\dr,\kl,\kr,\kk$ are invertible,
then
the symplectic blob algebra 
$b_{n}$  
is isomorphic to the algebra
$P_n$ via an isomorphism 
$$
\phi: P_n \rightarrow b_{n}
$$ 
induced by $E_0\mapsto e$, $E_1\mapsto e_1$,
$\ldots$, $E_{n-1}\mapsto e_{n-1}$ and $E_n\mapsto f$.
\end{thm}
It is straightforward to check that the generators already given for 
the symplectic
blob algebra satisfy the $P_n$ relations. Thus the map $\phi$ in the
theorem is a surjective homomorphism and hence we 
need only to prove injectivity.
The rest of this paper is devoted to proving this theorem.

\section{Definitions associated to the monoid $M_n$}

Two monomials $\bu,\bu'$ in the generators $S_n$ are said to be \emph{commutation equivalent}
if $\bu\equiv\bu'$ 
in $M_n$. 
The {\it commutation class}, $\overline{\bu}$, of a 
monomial $\bu$ consists of the monomials that are commutation equivalent to it.

The \emph{left descent set} (respectively, \emph{right descent set})
of a monomial $\bu$
consists of all the initial (respectively, terminal) letters of the 
elements of $\overline{\bu}$.  We denote these sets by $\ldescent{\bu}$ and
$\rdescent{\bu}$, respectively.

\begin{defn}\label{seca1}
A \emph{reduced monomial} is a monomial $\bu$ in the generators $S_n$
such that no $\bu' \in \overline{\bu}$ can be expressed
as a scalar multiple of a strictly shorter monomial
using the relations in Definition~\ref{2BTL pres}. 
\\
If we have $\bu = \bu_1 s \bu_2 s \bu_3$ for some generator
$s$, then the occurrences of $s$ in $\bu$ are said to be \emph{consecutive} if
$\bu_2$ contains no occurrence of $s$.
\end{defn}

\begin{defn}\label{seca2}
Two monomials in the generators, $\bu$ and $\bu'$, are said to be
\emph{weakly equivalent} if $\bu$ can be transformed into a nonzero multiple of
$\bu'$ by applying finitely many relations in $P_n$.  
\\
In this situation, we also say
that $D$ and $D'$ are weakly equivalent, where $D$ and $D'$ are the diagrams
equal to $\phi(\bu)$ and $\phi(\bu')$, respectively.  
If $P$ is a property that 
diagrams may or may not possess, then we say $P$ is \emph{invariant under weak
equivalence} if, whenever $D$ and $D'$ are weakly equivalent diagrams, then
$D$ has $P$ if and only if $D'$ has $P$.
\end{defn}

\begin{defn}\label{seca3}
Let $D$ be a diagram.  For $g \in \{L, R\}$ and $$
k \in \{1, \ldots, n, 1', \ldots, n'\}
,$$ we say that $D$ is \emph{$g$-decorated at the point $k$} if (a) the edge
$x$ connected to $k$ has a decoration of type $g$, and (b) the decoration of
$x$ mentioned in (a) is closer to point $k$ than any other decoration on $x$.
\end{defn}

In the sequel, we will sometimes invoke Lemma \ref{seca4} without explicit comment.

\begin{lem}\label{seca4}
The following properties of diagrams are invariant under weak equivalence:
\begin{enumerate}
\item[(i)]{the property of being $L$-decorated at the point $k$;}
\item[(ii)]{the property of being $R$-decorated at the point $k$;}
\item[(iii)]{for fixed $1 \leq i < n$, the property of points $i$ 
and $(i+1)$ being connected by an undecorated edge;}
\item[(iv)]{for fixed $1 \leq i < n$, the property of points $i'$ 
and $(i+1)'$ being connected by an undecorated edge.}
\end{enumerate}
\end{lem}

\begin{proof}
It is enough to check that each of these properties is respected by each type
of diagrammatic reduction, because the diagrammatic algebra is a homomorphic
image of the algebra given by the monomial presentation.
This presents no problems, but notice that
the term ``undecorated'' cannot be removed from parts (iii) and (iv), because
of the topological relation.
\end{proof}

Elements of the commutation monoid 
$M_n$ have the following normal form, established in \cite{CF}.

\begin{prop}[Cartier--Foata normal form]\label{seca6} 
Let $\bs$ be an element of the commutation monoid $M_n$.  Then $\bs$ has
a unique factorization in $M_n$ of the form 
$$
\bs = \bs_1 \bs_2 \cdots \bs_p
$$ 
such that each $\bs_i$ is a product of distinct commuting elements of
$S_n$, and such that for each $1 \leq j < p$ and each generator $t \in S_n$
occurring in $\bs_{j+1}$, there is a generator $s \in S_n$ occurring in
$\bs_j$ such that $st \ne ts$ or $s=t$.
\end{prop}

\begin{rem}\label{seca7}
The Cartier--Foata normal form may be defined inductively, as follows.  
Let $\bs_1$ be the product of
the elements in $\ldescent{\bs}$.  Since $M_n$ is a cancellative monoid,
there is a unique element $\bs' \in M_n$ with $\bs = \bs_1 \bs'$.  If $$
\bs' = \bs_2 \cdots \bs_p
$$ is the Cartier--Foata normal form of $\bs'$, then $$
\bs_1 \bs_2 \cdots \bs_p
$$ is the Cartier--Foata normal form of $\bs$.
\end{rem}

\begin{defn}\label{secc1}
Let $\bu$ be a reduced monomial in the generators $E_0, \ldots, E_n$.  We say
that $\bu$ is \emph{left reducible} (respectively, \emph{right reducible})
if it is commutation equivalent to a monomial of the form $\bu' = st \bv$ 
(respectively, $\bu' = \bv ts$), where $s$ and $t$ are noncommuting
generators and $t \not\in \{E_0, E_n\}$.  In this situation, we say that
$\bu$ is left (respectively, right) reducible via $s$ to $t \bv$ (respectively,
to $\bv t$).
\end{defn}

\section{Preparatory lemmas}   %%%Diagram reductions} 

The following result is similar to \cite[Lemma 5.3]{G35}, but we give
a complete argument here because the proof in
\cite{G35} contains a mistake (we thank D. C. Ernst for pointing this out).

\begin{lem}\label{secc2}
Suppose that $\bs \in M_n$ corresponds to a reduced monomial, 
and let $\bs_1 \bs_2 \cdots \bs_p$ be the Cartier--Foata
normal form of $\bs$.  Suppose also that $\bs$ is not left reducible.
Then, for $1 \leq i < p$ and $0 \leq j \leq n$, the following hold:
\begin{enumerate}
\item[(i)]
{if $E_0$ occurs in $\bs_{i+1}$, then $E_1$ occurs in $\bs_i$;}
\item[(ii)]
{if $E_n$ occurs in $\bs_{i+1}$, then $E_{n-1}$ occurs in $\bs_i$;}
\item[(iii)]
{if $j \not\in \{0, n\}$ and $E_j$ occurs in $\bs_{i+1}$, then both 
$E_{j-1}$ and $E_{j+1}$ occur in $\bs_i$.}
\end{enumerate}
\end{lem}

\begin{proof}
The assertions of (i) and (ii) are almost immediate from properties of the normal
form, because $E_1$ (respectively, $E_{n-1}$) is the only generator not
commuting with $E_0$ (respectively, $E_n$). We need only consider the
other 
alternative of $E_0$ (respectively, $E_n$) being in $\bs_i$.
If $E_0$ occurs in $\bs_i$
then this moves to the end of the $\bs_i$ which then cancels with the
$E_0$ from the $\bs_{i+1}$ contradicting the assumption that $\bs$ is
not reducible.

We will now prove (iii) 
by induction on $i$.  Suppose first that $i = 1$.

Suppose that $j \not\in \{0, n\}$ and that $E_j$ occurs in $\bs_2$.  
By definition of the normal form, there must be a generator $s \in \bs_1$ 
that does not commute with $E_j$ or $s=E_j$.  If $s=E_j$ then $\bs$ is
reducible as before. If $s \ne E_j$ then $s$ cannot be the only generator
that does not commute with $E_j$, or 
$\bs$ would be left reducible via $s$.
Since the only generators not
commuting with $E_j$ are $E_{j-1}$ and $E_{j+1}$, these must both
occur in $\bs_1$.

Suppose now that the statement is known to be true for $i < N$, and
let $i = N \geq 2$.
Suppose also that $j \not\in \{0, n\}$ and that $E_j$ occurs in $\bs_{N+1}$.  
As in the base case, there must be at least one generator $s$ occurring 
in $\bs_N$ that does not commute with $E_j$.  

Let us first consider the case where $j \not\in \{1, n-1\}$, and write
$s = E_k$ for some $0 \leq k \leq n$.  The restrictions on $j$ means
that $k \not \in \{0, n\}$ and so that
we cannot have $E_j E_k E_j$ occurring as a subword of any reduced monomial.
However, $E_j$ occurs in $\bs_{N-1}$ by the inductive hypothesis, and this
is only possible if there is another generator, $s'$, in $\bs_N$ that does not 
commute with $E_j$.  This implies that $\{s', E_k\} = \{E_{j-1}, E_{j+1}\}$,
as required.

Now suppose that $j = 1$ (the case $j = n-1$ follows by a symmetrical
argument).  If both $E_0$ and $E_2$ occur in $\bs_N$, then
we are done.  If $E_2$ occurs in $\bs_N$ but $E_0$ does not, then the
argument of the previous paragraph applies.  Suppose then that $E_0$ occurs
in $\bs_N$ but $E_2$ does not.  By statement (i), $E_1$ occurs in 
$\bs_{N-1}$, but arguing as in the previous paragraph, we find this cannot
happen, because it would imply that $\bs$ was commutation equivalent to
a monomial of the form $\bv' E_1 E_0 E_1 \bv''$, which is incompatible
with $\bs$ being reduced.  This completes the inductive step.
\end{proof}

The following is a key structural property of reduced monomials.

\begin{prop}\label{secc3}
Suppose that $\bs \in M_n$ corresponds to a reduced monomial, 
and let $\bs_1 \bs_2 \cdots \bs_p$ be the Cartier--Foata
normal form of $\bs$, where $\bs_p$ is nonempty.
Suppose also that $\bs$ is neither left reducible nor right reducible.
Then either {\rm{(i)}} $p = 1$, meaning that $\bs$ is a product of commuting
generators or {\rm{(ii)}} $p = 2$ and either $\bs = IJ$ or $\bs = JI$.
\end{prop}

\begin{proof}
If $p = 1$, then case (i) must hold, so we will assume
that $p > 1$.  

A consequence of Lemma \ref{secc2} is that 
if $\bs_{i+1} = I$ then $\bs_i = J$, and 
if $\bs_{i+1} = J$ then $\bs_i = I$.
It follows that if $\bs_p \in \{I, J\}$ (in $P_n$), then 
$\bs$ must be an alternating
product of $I$ and $J$.  Since $\bs$ is reduced, this forces $p = 2$ and
either $\bs = IJ$ or $\bs = JI$.  We may therefore assume that $\bs_p \not\in
\{I, J\}$.

Since $\bs_p \not \in \{I, J\}$ and $\bs_p$ is a product of commuting 
generators, at least one of the following two situations must occur.
\begin{enumerate}
\item[(a)]{For some $2 \leq i \leq n$, $\bs_p$ contains an occurrence of $E_i$
but not an occurrence of $E_{i-2}$.}
\item[(b)]{For some $0 \leq i \leq n-2$, $\bs_p$ contains an occurrence of 
$E_i$ but not an occurrence of $E_{i+2}$.}
\end{enumerate}

Suppose we are in case (a).  In this case, Lemma \ref{secc2} 
means that there must
be an occurrence of $E_{i-1}$ in $\bs_{p-1}$; 
Now $E_{i-1}$ fails to commute with two other generators ($E_i$
and $E_{i-2}$).  However,
one of these generators, $E_{i-2}$ does not occur in $\bs_p$.  It follows 
that $\bs$ is right reducible (via $E_i$), which is a contradiction.
Case (b) leads to a similar contradiction, again involving right
reducibility, which completes the proof.
\end{proof}

\begin{lem}\label{secb1}
Let $\bu = \bu_1 s \bu_2 s \bu_3$ be a reduced word in which the occurrences
of the generator $s$ are consecutive, and suppose that every generator
in $\bu_2$ not commuting with $s$ is of the same type, $t$ say.  Then 
$\bu_2$ contains only one occurrence of $t$, and $s \in \{E_0, E_n\}$.
\end{lem}

\begin{proof}
The proof is by induction on the length, $l$, of the word $\bu_2$.  Note
that $\bu_2$ must contain at least one generator not commuting with $s$,
or after commutations, we could produce a subword of the form $ss$.
This means that the case $l = 0$ cannot occur.  

If $\bu_2$ contains only one generator not commuting with $s$, then after
commutations, $\bu$ contains a subword of the
form $sts$.  This is only possible if $s \in \{E_0, E_n\}$, and this establishes
the case $l = 1$ as a special case.

Suppose now that $l > 1$.  By the above paragraph, we may reduce to the case
where $\bu_2 = \bu_4 t \bu_5 t \bu_6$, and the indicated
occurrences of $t$ are consecutive. Since the occurences of $s$ were
consecutive, $\bu_5$ does not contain $s$. Thus every generator in
$\bu_5$ that does not commute with $t$ will be the same generator $u
\ne s$. Now since $\bu_5$ is shorter than
$\bu_2$, we can apply the inductive hypothesis to show that $t \in
\{E_0, E_n\}$
and $\bu_5$ contains only one occurrence of the generator, $u$.
But this means that $t$ fails to commute with two different
generators, $u$ and $s$ 
contradicting the fact that $t \in \{E_0, E_n\}$ and completing
the proof.
\end{proof}

\begin{lem}\label{secb2}
Let $\bu$ be a reduced monomial.
\begin{enumerate}
\item[(i)]{Between any two consecutive occurrences of $E_0$ in $\bu$,
there is precisely one letter not commuting with $E_0$ (\idest an occurrence
of $E_1$).}
\item[(ii)]{Between any two consecutive occurrences of $E_n$ in $\bu$,
there is precisely one letter not commuting with $E_n$ (\idest an occurrence
of $E_{n-1}$).}
\item[(iii)]{Let $0 < i < n$. Between any two consecutive occurrences
  of $E_i$ in $\bu$,
there are precisely two letters not commuting with $E_i$, and they correspond
to distinct generators.}
\end{enumerate}
\end{lem}

\begin{proof}
To prove (i), we apply Lemma \ref{secb1} with $s = E_0$; the hypotheses
are satisfied as we necessarily have $t = E_1$.  The proof of (ii) is similar.

To prove (iii), write $\bu = \bu_1 s \bu_2 s \bu_3$ for consecutive
occurrences of the generator $s = E_i$.  Since $s \not\in \{E_0, E_n\}$, the 
hypotheses of Lemma \ref{secb1} cannot be satisfied, so $\bu_2$ must have at
least one occurrence of each of $t_1 = E_{i-1}$ and $t_2 = E_{i+1}$.
Suppose that $\bu_2$ contains two or more occurrences of $t_1$.  The fact
that the occurrences of $s$ are consecutive means that two consecutive
occurrences of $t_1$ cannot have an occurrence of $s$ between them.
Applying Lemma \ref{secb1}, this means that there is precisely one generator
$u$ between the consecutive occurrences of $t_1$ such that $t_1 u \ne u t_1$,
and furthermore, that $t_1 \in \{E_0, E_n\}$.  This is a contradiction, because
$t_1$ fails to commute with two different generators ($s$ and $u$).

One can show similarly that $\bu_2$ cannot contain two or more 
occurrences of $t_2$.  We conclude that each of $t_1$ and $t_2$ occurs
precisely once, as required.
\end{proof}

\section{The map $\phi$}

Recall the map  $\phi: P_n \to b_n$ from Theorem \ref{thmmain}.
Here we will consider the possible diagrams arising from reduced
monomials in $P_n$.
We let $D_{\bs}$ be the `concrete' pseudo-diagram \cite{gensymp} associated to a 
 monomial
$\bs=E_{i_1} E_{i_2} \cdots E_{i_m}$
formed by concatenating 
$e_{i_1}$,  $e_{i_2}$,  $\ldots$,  $e_{i_m}$ in order but without
applying any straightening, 
and without applying any further isotopies that deform across the
bounding frames of the concatenating components. Thus we  include the
possiblity that $D_{\bs}$ has loops.
So $D_{\bs} = \phi(\bs) $ as (a scalar multiple of) a diagram, after applying any
straightening rules, but the shape of
the concrete pseudo-diagram $D_{\bs}$ allows us to reconstruct $\bs$.
For example, the monomial, $\bs = E_1E_2E_4E_0E_1$ has
(concrete pseudo-)diagram as illustrated in Figure~\ref{fig0}.
\begin{figure}[ht]
$$
D_\bs = e_1e_2e_4e_0e_1 =
\raisebox{-2.5cm}{\epsfbox{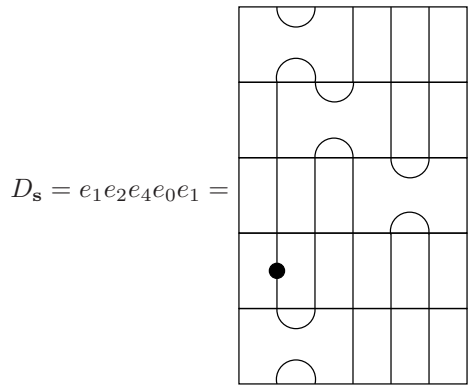}}
$$
\caption{Concrete pseudo-diagram associated to the monomial 
$\bs = E_1E_2E_4E_0E_1$}\label{fig0}
\end{figure}

The non-loop arcs in the concrete pseudo-diagram $D_{\bs}$ are made up 
of vertical line
segments, cups and caps.  In the following development, we will regard such arcs
as having a direction or orientation (as we shall see shortly, this
arc orientation can be chosen arbitrarily). 
Thus each vertical line segment becomes
oriented northwards ($N$) or southwards ($S$), and each cup or cap is oriented
westwards ($W$) or eastwards ($E$).  
If an oriented arc contains an occurrence of $E$ after
an occurrence of $W$, then we say that the arc has a {\it west-east direction
reversal}; we define {\it east-west direction reversal} analogously.
If an arc has at least one direction reversal, then we say that the arc
{\it changes direction}.

Consider, for example, the decorated arc in Figure~\ref{fig0}.  If we orient
the topmost vertical line segment in this arc by $S$, then, starting with this
line segment and working in the direction of the orientation, the segments of 
this arc are consecutively labelled $$
S, W, W, S, S, S, E, N, N, E, S, S, S
,$$ where the sixth letter from the left (an occurrence of $S$) corresponds 
to the decorated segment.
On the other hand, if we orient the topmost vertical line segment by $N$, then
the segments of the arc are consecutively labelled $$
N, N, N, W, S, S, W, N, N, N, E, E, N
$$ in the direction of the orientation, and the eighth letter from the left
(an occurrence of $N$) corresponds to the decorated segment.
In both cases, we have a west-east direction reversal, but no east-west
direction reversal, and the arc changes direction.  

The above example illustrates the basic fact that the property of having a 
west-east direction reversal does not depend on the orientation chosen 
for an arc.  For similar reasons, the same can be said about east-west 
direction reversals, and about the property of changing direction.

It will turn out to be significant (see Lemma~\ref{secb3} below) that 
between any occurrence of $W$ and an occurrence of $E$ in the particular
arc of Figure~\ref{fig0} studied above, there
is a vertical segment corresponding to a decoration in the diagram.  
Figure~\ref{fignoneg} contains a west-east reversal in which this does
not happen, but the concrete pseudo-diagram in Figure~\ref{fignoneg} corresponds 
to the non-reduced monomial $E_1 E_2 E_3 E_1$.

\begin{figure}[ht]
\epsfbox{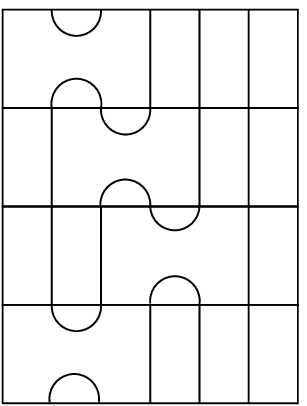}
\caption{West-east direction reversal of an undecorated
  line.}\label{fignoneg}
\end{figure}

\begin{lem}\label{secb3}
Let $D$ be a diagram of the form $\phi(\bs)$ for some reduced monomial $\bs$.
\begin{enumerate}
\item[(i)]
If an arc of $D$ contains a west-east direction reversal, then that arc must
contain a consecutive sequence $X_1, \ldots, X_k$ of cups, caps and vertical
segments with $1 < a < b < k$ and $a < b-1$, such that
\begin{enumerate}
\item[(a)] $X_1$ and $X_a$ are labelled $W$;
\item[(b)] $X_b$ and $X_k$ are labelled $E$;
\item[(c)] for some $P$, $Q$ with $\{P, Q\} = \{N, S\}$, 
\begin{enumerate}
\item[(1)] the $X_i$ for $a < i < b$ are all labelled $P$, and exactly one
of them carries a left blob;
\item[(2)] the $X_j$ for $1 < j < a$ and for $b < j < k$ are all labelled $Q$;
\end{enumerate}
\item[(d)] $X_1$ and $X_k$ form a cup-cap pair corresponding to a single
occurrence of $E_1$.
\end{enumerate}
\item[(ii)]
If an arc of $D$ contains a west-east direction reversal, then that arc must
contain a consecutive sequence $X_1, \ldots, X_k$ of cups, caps and vertical
segments with $1 < a < b < k$ and $a < b-1$, such that
\begin{enumerate}
\item[(a)] $X_1$ and $X_a$ are labelled $E$;
\item[(b)] $X_b$ and $X_k$ are labelled $W$;
\item[(c)] for some $P$, $Q$ with $\{P, Q\} = \{N, S\}$, 
\begin{enumerate}
\item[(1)] the $X_i$ for $a < i < b$ are all labelled $P$, and exactly one
of them carries a right blob;
\item[(2)] the $X_j$ for $1 < j < a$ and for $b < j < k$ are all labelled $Q$;
\end{enumerate}
\item[(d)] $X_1$ and $X_k$ form a cup-cap pair corresponding to a single
occurrence of $E_{n-1}$.
\end{enumerate}
\end{enumerate}
\end{lem}

\begin{proof}
Recall that Lemma~\ref{secb2} constrains what can happen
between two occurrences of the same generator $E_i$ in the reduced monomial 
$\bs$.  Up to commutation of 
generators, these cases are shown in the next five diagrams, which illustrate
the cases $i = 0$, $i = n$, $2 \leq i \leq n-2$, $i = 1$ and $i = n-1$
respectively.
$$
ee_1e =
\raisebox{-1.5cm}{\epsfbox{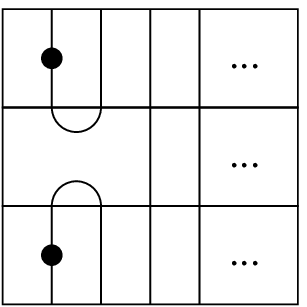}},
\qquad
fe_{n-1}f =
\raisebox{-1.5cm}{\epsfbox{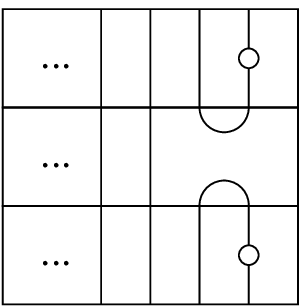}}
$$
$$
e_ie_{i-1}e_{i+1}e_i =
\raisebox{-2.0cm}{\epsfbox{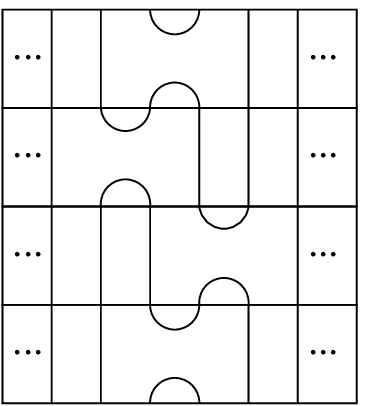}}
, \quad 2 \le i \le n-2,
$$
$$
e_1ee_{2}e_1 =
\raisebox{-2.0cm}{\epsfbox{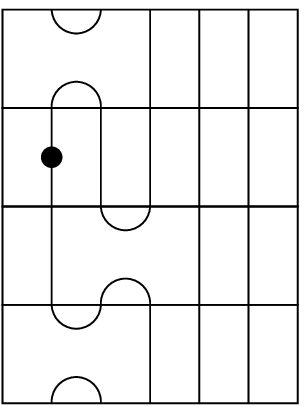}},
\qquad
e_{n-1}fe_{n-2}e_{n-1} =
\raisebox{-2.0cm}{\epsfbox{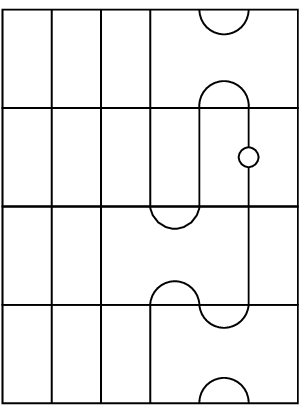}}
$$

If a direction reversal occurs in an arc of $D$, this must correspond to
a consecutive sequence of oriented segments of the form $$
Y_1, Y_2, \ldots, Y_{k-1}, Y_r
$$ in which $\{Y_1, Y_r\} = \{W, E\}$, and all the $Y_i$ for
$1 < i < k$ are either all equal to $S$ or all equal to $N$.  In this case,
$Y_1$ and $Y_r$ correspond to distinct letters of $\bs$; call these 
$\bs_1$ and $\bs_2$.  Because $Y_1$ and $Y_2$ are separated only by vertical
line segments, it follows that $\bs_1$ and $\bs_2$ correspond to
consecutive occurrences in $\bs$ of the same generator, $E_i$.  The 
possibilities enumerated in the diagrammatic version of Lemma~\ref{secb2}
now force either $i = 1$ or $i = n-1$, and the conclusions now follow from the
corresponding two pictures, with $X_a = Y_1$ and $X_b = Y_r$.
\end{proof}

\begin{lem}\label{secb4}
Let $D$ be a diagram representing a reduced monomial $\bu$ (\idest
$D=\phi(\bu)$).
\begin{enumerate}
\item[(i)]{The diagram $D$ is $L$-decorated at $1$ (respectively, $1'$)
if and only if 
the left (respectively, right) descent set of $\bu$ contains $E_0$.}
\item[(ii)]{The diagram $D$ is $R$-decorated at $n$ (respectively, $n'$), 
if and only if 
the left (respectively, right) descent set of $\bu$ contains $E_n$.}
\item[(iii)]{Suppose that $1 \leq i < n$.  Then points $i$ and $i+1$ 
(respectively, $i'$ and $(i+1)'$) in $D$ are connected by an undecorated 
edge if and only if the left
(respectively, right) descent set of $\bu$ contains $E_i$.}
\end{enumerate}
\end{lem}

\begin{proof}
In all three cases, the ``if'' statements follow easily from diagram calculus
considerations, so we only prove the ``only if'' statements.

Suppose for a contradiction that $D$ is $L$-decorated at $1$, but that the
left descent set of $\bu$ does not contain $E_0$. 
For this to happen, the 
line leaving point $1$ must eventually encounter an $L$-decoration, but must
first encounter a cup corresponding to an occurrence of the
generator $E_1$.  The only way this can happen and be consistent with Lemma
\ref{secb3} is for the 
line to then travel to the east wall after encountering
$E_1$, then change direction and then travel back to the west wall, as shown
in Figure \ref{fig2}.  (Note that this can only happen if $n$ is odd, and that
as before, 
the thin dotted lines in the diagram indicate pairs of horizontal edges
that correspond to the same generator. 
\begin{figure}
        %\epsfbox{dum3.eps}
        \input{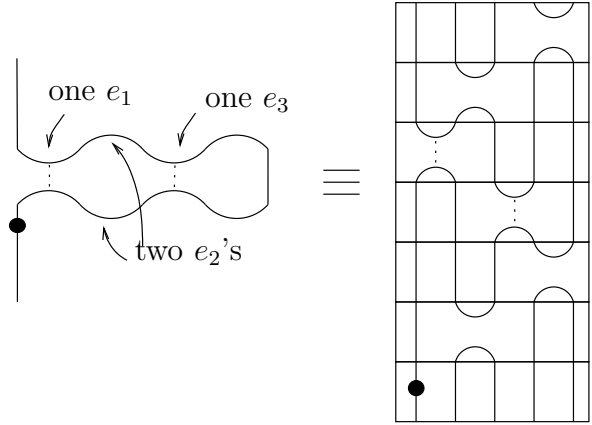}
\caption{Illustrating the proof of Lemma \ref{secb4}~(i).\label{fig2}}
\end{figure}
In the figure we illustrate the element $E_4E_2E_1E_3E_4E_2E_0$ (which is
not actually reduced).)
So we may obtain another reduced expression for $\bu$, namely, 
$\bu = \bv E_1 \bu'' E_0 \bu'''$ where $\bu''$ and
$\bu'''$ are  reduced, and $\bv E_1\bu''$ does not contain $E_0$.

The arc leaving $1$ in the diagram for $\bv E_1 \bu''$ contains an east-west
direction reversal, and thus by Lemma~\ref{secb3} contains an occurrence
of $E_n$, but no occurrence of $E_0$.  This arc is therefore $R$-decorated.
By Lemma~\ref{seca4}, neither it, nor the arc leaving $1$ from $D$, can be
$L$-decorated, which is a contradiction.

The claim regarding $1'$ and the right descent
set is proved similarly.  This completes the proof of (i), and the proof of
(ii) follows by modifying the above proof in the obvious way.

We now turn to (iii).
Suppose for a contradiction that points $i$ and $i+1$ in $D$ are connected
by an undecorated edge, but that the
left descent set of $\bu$ does not contain $E_i$.  

Since $i$ is connected to $i+1$,
the arc leaving $i$ (respectively, $i+1$) must (after possibly traversing
some vertical line 
segments) either encounter a cup corresponding to $E_{i-1}$ or $E_i$
(respectively, $E_i$ or $E_{i+1}$).  Suppose for a contradiction that
the arc leaving $i$ encounters an $E_{i-1}$ first.  By Lemma~\ref{secb3},
the arc leaving $i$ performs a west-east direction reversal, as shown
in Figure~\ref{fig3}.
\begin{figure}
\caption{Illustrating the proof of Lemma \ref{secb4}~(iii).\label{fig3}}
        \epsfbox{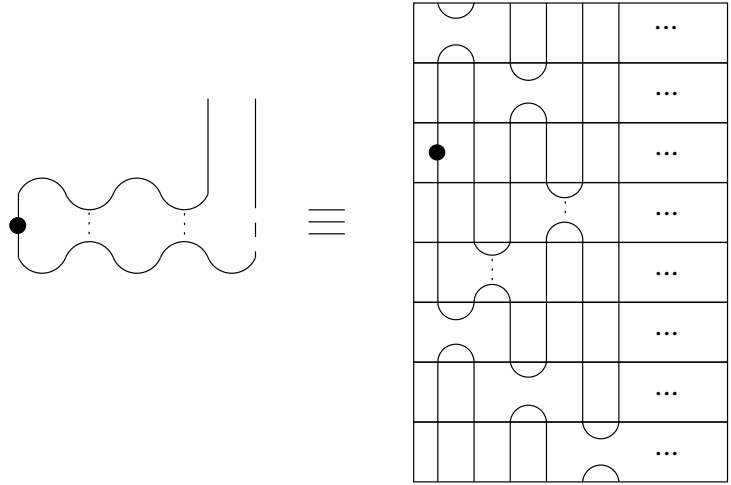}
\end{figure}
This implies that the arc is $L$-decorated at $i$, which contradicts
Lemma~\ref{seca4}.

We have shown that the arc leaving $i$ encounters an occurrence of $E_i$ first,
and a similar argument shows that the arc leaving $i+1$ also encounters an
occurrence of $E_i$ first.  This proves the claim about the left descent
set.
The claim regarding the right descent
set is proved similarly.  This completes the proof of (iii).
\end{proof}

\begin{lem}\label{secb5}
Let $\bu$ and $\bu'$ be reduced monomials that map to the same
diagram $D$ under $\phi$.
\begin{enumerate}
\item[(i)]{If $\bu'$ is a product of commuting generators, then
$\bu$ and $\bu'$ are equal in $P_n$.}
\item[(ii)]{If $\bu' = IJ$ or $\bu' = JI$, then $\bu$ and $\bu'$
are equal in $P_n$.}
\end{enumerate}
\end{lem}

\begin{proof}
Note that as $\bu'$ is the product of commuting generators we have
$$\ldescent{\bu'} = \rdescent{\bu'}=\{E_i \mid E_i \mbox{ occurs in a
  minimal length expression for } \bu'\}$$

We first prove (i).  By Lemma \ref{secb4} and the fact that $\bu$ and $\bu'$
represent the same diagram, we must have 
$$
\ldescent{\bu} = \ldescent{\bu'} = \rdescent{\bu'} = \rdescent{\bu}.
$$

Suppose that $\bu'$ contains an occurrence of the generator $E_0$.
This implies that $\bu$ must contain an occurrence of $E_0$, because  
$E_0 \in \ldescent{\bu'} = \ldescent{\bu}$.
Suppose also (for a contradiction) that $\bu$ contains two occurrences of
the generator $E_0$.  By Lemma \ref{secb2}, there must be an occurrence of
$E_1$ between the first (\idest leftmost or northernmost) two occurrences
of $E_0$.  

Since points $1$ and $1'$ of $D$ are connected by an $L$-decorated
line,
(using Lemma \ref{secb4})
there must be an occurrence of $E_2$ immediately above the aforementioned 
occurrence of $E_1$ in order to prevent the line emerging from $1$ 
from exiting the diagram at $2$, as illustrated below:
$$
\epsfbox{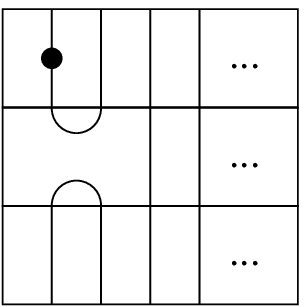}
$$  (``Immediately above'' means that there
are no other occurrences of $E_1$ or $E_2$ between the two occurrences
mentioned.)  In turn, we must have an occurrence of $E_3$ immediately below
the aforementioned occurrence of $E_2$ in order to prevent the line from
exiting the box at point $3'$.  
This is only sustainable if the arc between $1$ and $1'$ has an east-west
direction reversal.  Lemma~\ref{secb4} then forces the arc to contain
a right blob, which in turn implies that $n$ is odd, as shown below.
$$ 
       \epsfbox{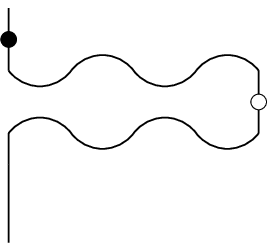}
$$

There are two ways this picture can continue to the bottom.
Either the line exits the box at point 
$1'$ without encountering further generators, or the line encounters an 
occurrence of $E_1$.  The first situation cannot occur because it contradicts
Lemma \ref{secb4} and the hypothesis that $E_0 \in \rdescent{\bu}$.  The second
situation cannot occur because it shows (by repeating the 
argument in the paragraph above)
that $\bu$ is commutation equivalent
to a monomial of the form $\bv JIJ \bv'$, which contradicts the hypothesis
that $\bu$ be reduced.

We conclude that $\bu$ contains precisely
one occurrence of $E_0$, and furthermore that 
$\bu$ contains no occurrences of $E_1$.  

A similar argument shows that if $\bu'$ contains an occurrence of the 
generator $E_n$, then $\bu$ contains at most one occurrence of $E_n$, and it 
can only contain $E_n$ if it contains no occurrences of $E_{n-1}$.

It follows that at least one of the three situations must occur:
\begin{enumerate}
\item[(a)]{$\bu'$ contains $E_0$ and $\bu = E_0 D_E$, 
where $D_E$ contains no occurrences of $E_0$ or $E_1$;}
\item[(b)]{$\bu'$ contains $E_n$ and $\bu = D_F E_n$, 
where $D_F$ contains no occurrences of 
$E_{n-1}$ or $E_n$;}
\item[(c)]{$\bu'$ contains neither $E_0$ nor $E_n$.}
\end{enumerate}

In cases (a) and (b), there is a corresponding factorization of $\bu'$, 
and the result claimed now follows from the faithfulness of the 
diagram calculus for the blob algebra \cite{blobcgm,marblobpres}.
For example, in case (a), we have
$\bu' = E_0 D'_E$ and $\bu = E_0 D_E$. We view $D'_E$ and $D_E$ as
elements of the blob algebra, where the blob in this 
case is identified with $E_n$.
Then as $D'_E$ and $D_E$ have the same diagram, they must be also
equal in $P_n$ by the faithfulness of the blob algebra.
Thus $\bu'$ and $\bu$ are equal in $P_n$.

Suppose that we are in case (c), but that $\bu$ contains an occurrence
of $E_0$ or $E_n$.  Because the diagram $D$ corresponds to $\bu'$, it cannot 
have decorations, so it must be the case that $\phi(\bu)$ is either $L$-decorated at some
point, or $R$-decorated at some point.  This contradicts the hypotheses
on $\bu'$, using Lemma \ref{seca4}.  Since neither $\bu$ nor $\bu'$ contains
$E_0$ or $E_n$, the result follows by the faithfulness of the diagram calculus
for the Temperley--Lieb algebra 
\cite[\S6.4]{marbk}.
This completes the proof of (i).

We now prove (ii) in the case where $\bu' = IJ$; the case $\bu' = JI$
follows by a symmetrical argument.  
Thus, $\bu$ maps to the 
same diagram as $IJ$.  The fact that 
$\ldescent{\bu}$ is the set of generators in $I$ and $\rdescent{\bu}$ 
is the set of generators 
in $J$ means that $\bu$ cannot be left or right reducible.  
By Proposition \ref{secc3} (ii), this 
immediately means that $\bu = IJ$.
\end{proof}

\section{Proof of the theorem} 

\begin{lem}\label{secc4}
Let $\bu$ be a reduced monomial and let $D$ be the corresponding diagram.
Then $D$ avoids all the features on the left hand sides of Table
\ref{blobtab}, (the table in section \ref{sect:symp} depicting all the
straightening relations).
Furthermore, $D$ contains at most one line
with more than one decoration.
\end{lem}

\begin{proof}
The proof is by induction on the length of $\bu$.
If $\bu$ is a product of commuting generators, or $\bu = IJ$, or $\bu = JI$,
the assertions are easy to check, so we may assume that this is not the case.
(This covers the base case of the induction as a special case.)

By Proposition \ref{secc3}, $\bu$ must either be left reducible or right reducible.
We treat the case of left reducibility; the other follows by a symmetrical
argument.

By applying commutations to $\bu$ if necessary, we may now assume that 
$\bu = st \bv$, where $s$ and $t$ are noncommuting generators, and $t \not\in
\{E_0, E_n\}$.  By induction, we know that the reduced monomial $t \bv$ corresponds
to a diagram $D'$ with none of the forbidden features and at most one edge
with two decorations.

Suppose that $t = E_1$ and $s = E_0$.  By Lemma \ref{secb4}, points $1$ and $2$ of
$D'$ must be connected by an undecorated edge, and the effect of multiplying
by $E_0$ is simply to decorate this edge.  This does not introduce any forbidden
features, nor does it create an edge with two decorations, and this completes
the inductive step in this case.

The case where $t = E_{n-1}$ and $s = E_n$ is treated similarly to the above
case, so we may now assume that $s, t \not\in \{E_0, E_n\}$.  We must either have
$s = E_i$ and $t = E_{i+1}$, or vice versa.

Suppose that $s = E_i$ and $t = E_{i+1}$.  By
Lemma \ref{secb4}, this means that points $i+1$ and $i+2$ of $D'$ are connected 
by an undecorated edge.  The effect of multiplying by $s$ is then (a) to remove
this undecorated edge, then (b) to disconnect the edge emerging from point $i$ 
of $D'$ and reconnect it to point $i+2$, retaining its original decorated
status, then (c) to install an undecorated edge between points $i$ and
$i+1$.  This procedure does not create any forbidden features, nor does it
create a new edge with more than one decoration.

The case in which $s = E_{i+1}$ and $t = E_i$ is treated using a parallel
argument, and this completes the inductive step in all cases.
\end{proof}

\begin{lem}\label{secc5}
Let $\bu$ be a reduced monomial with corresponding diagram $D$.
\begin{enumerate}
\item[\rm (i)]{If points $1$ and $2$ (respectively, $1'$ and $2'$) 
are connected in $D$ by an edge decorated by $L$ but not $R$,
then $\bu$ is equal (as an algebra element) to a word of
the form $\bu' = E_0 E_1 \bv$ (respectively, $\bu' = \bv E_1 E_0$).}
\item[\rm (ii)]{If points $n-1$ and $n$ (respectively, $(n-1)'$ and $n'$)
are connected in $D$ by an edge
decorated by $R$ but not $L$,
then $\bu$ is equal (as an algebra element) to a word of
the form $\bu' = E_n E_{n-1} \bv$ (respectively, $\bu' = \bv E_{n-1} E_n$).}
\end{enumerate}
\end{lem}

\begin{proof}
We first prove the part of (i) dealing with points $1$ and $2$.
By Lemma \ref{secb4}, we have $E_0 \in \ldescent{\bu}$, so $\bu = E_0 \bv'$.
Now $\bv'$ is also a reduced monomial, and by Lemma \ref{secc4}, $\bv'$
corresponds to a diagram $D'$ with no forbidden features.  Since multiplication
by $e$ does not change the underlying shape of a diagram (ignoring the
decorations), it must be the case that points $1$ and $2$ of $D'$ are 
connected by an edge with some kind of decoration.
Since $D$ has no forbidden features and
the corresponding edge in $D$ has no $R$-decoration, the only way for this
to happen is if the edge connecting points $1$ and $2$ in $D'$ is undecorated.
By Lemma \ref{secb4}, this means that $\bv'$ is equal as an algebra element to
a monomial of the form $E_1 \bv$, and this completes the proof of (i) in
this case.

The other assertion of (i) and the assertions of (ii) follow by parallel 
arguments.
\end{proof}

\begin{proof}[Proof of Theorem \ref{thmmain}]
As all the parameters are invertible, it is enough to prove
the statement when $\kl=\kr=1$.
Indeed the  rescaling of
$E_0$ and $E_n$:
$$
E_0\mapsto \frac{E_0}{\kl},\quad  E_n\mapsto \frac{E_n}{\kr}
$$
 has the following effect on the parameters:
$$
\dd\mapsto \dd,\quad
\dl\mapsto \frac{\dl}{\kl},\quad
\dr\mapsto \frac{\dr}{\kr},\quad
\kl\mapsto 1,\quad
\kr\mapsto 1,\quad
\kk\mapsto \frac{\kk}{\kl\kr}
$$
Thus 
any $P_n$ (with 6 parameters)
is isomorphic to a case with $\kl=\kr=1$.

It is clear from the generators and relations that the reduced monomials are
a spanning set, and that the diagram algebra is a homomorphic image of the
abstractly defined algebra.  By Lemma \ref{secc4}, all reduced monomials map to
basis diagrams.  The only way the homomorphism could fail to be injective
is therefore for two reduced monomials $\bu$ and $\bu'$ to map to the same
diagram $D$, and yet to be distinct as elements in $P_n$.  

It is therefore enough to prove that if $\bu$ and $\bu'$ are reduced monomials
mapping to the same diagram, then they are equal in $P_n$.
Without loss of generality, we assume 
that $\ell(\bu) \leq \ell(\bu')$ (where $\ell$ denotes length).

We proceed by induction on $\ell(\bu)$.  If $\ell(\bu) \leq 1$, or, more
generally, if $\bu$ is a product of commuting generators, then Lemma
\ref{secb5}
shows that $\bu = \bu'$.  Similarly, if $\bu = IJ$ or $\bu = JI$, then
$\bu = \bu'$, again by Lemma \ref{secb5}.  In particular, 
this deals with the base 
case of the induction.

By Proposition \ref{secc3}, we may now assume that $\bu$ is either 
left or right
reducible.  We treat the case of left reducibility, the other being similar.
By applying commutations if necessary, we may reduce to the case where
$\bu = st \bv$, $s$ and $t$ are noncommuting generators, and $t \not\in
\{E_0, E_n\}$.

Suppose that $s = E_0$, meaning that $t = E_1$.  In this case, points $1$ and
$2$ of $D$ are connected by an edge decorated by $L$ but not $R$.  By
Lemma \ref{secc5}~(i), this means that we have $\bu' = st \bv'$ as algebra
elements.
Since $\bu$ and $\bu'$ share a diagram, the (not necessarily reduced) monomials
$t \bu$ and $t \bu'$ must also share a diagram.  Since $tst = \kl t = t$, 
the (reduced) monomials $t \bv$ and $t \bv'$ also map to the same diagram,
$D'$.  However, $t \bv$ is shorter than $\bu$, so by induction, 
$t \bv = t \bv'$, which in turn implies that $\bu = \bu'$.

Suppose that $s = E_n$, meaning that $t = E_{n-1}$.  An argument similar to
the above, using Lemma \ref{secc5}~(ii), establishes that $\bu = \bu'$ in this
case too.

We are left with the case where $s = E_i$ and either $t = E_{i+1}$ or
$t = E_{i-1}$ (where $t \not\in \{E_0, E_n\}$).  We will treat the case where
$t = E_{i+1}$; the other case follows similarly.  In this case, we have
$tst = t$, and so $t \bu = tst \bv = t \bv$.  It is not necessarily true that
$t \bu'$ is a reduced monomial, but it maps to the same diagram as $t \bv$,
which is reduced.  After applying algebra relations to 
$t \bu'$, we may transform it into a scalar multiple of a reduced monomial,
$\br$.  Since reduced monomials map to basis diagrams (Lemma \ref{secc4}),
the scalar involved must be $1$.  Now the reduced monomials $t \bv$ and
$\br$ map to the same basis diagram, and $t \bv$ is shorter than $\bu$,
so by induction, we have $t \bv = \br$ in $P_n$.

Since $s \in \ldescent{\bu}$, we have $s \in \ldescent{\bu'}$ by Lemma
\ref{secb4}, so that $\bu' = s \bv''$ for some reduced monomial $\bv''$.
Since $sts = s$, we have $s (t \bu') = \bu'$.  We have shown that
$t \bu' = \br = t \bv$, so we have $$
\bu' = s t \bu' = st \bv = \bu
,$$ which completes the proof.
\end{proof}

\providecommand{\bysame}{\leavevmode\hbox to3em{\hrulefill}\thinspace}
\providecommand{\MR}{\relax\ifhmode\unskip\space\fi MR }
\providecommand{\MRhref}[2]{%
  \href{http://www.ams.org/mathscinet-getitem?mr=#1}{#2}
}
\providecommand{\href}[2]{#2}

\end{document}